\documentclass[11pt,a4paper]{amsart}

\usepackage[T1]{fontenc}
\usepackage[utf8]{inputenc}
\usepackage{lmodern}
\usepackage{microtype}
\usepackage{amsmath}
\usepackage{amssymb}
\usepackage{amsthm}
\usepackage[margin=1.15in]{geometry}
\usepackage{xcolor}
\usepackage{tikz}
\usetikzlibrary{shapes.geometric}
\usepackage{hyperref}
\hypersetup{colorlinks=true,linkcolor=blue!45!black,citecolor=blue!45!black,
            urlcolor=blue!45!black,
            pdftitle={Tuza's conjecture for graphs of maximum degree at most seven},
            pdfauthor={Anish Gupta},
            pdfsubject={Triangle packing and covering in bounded-degree graphs},
            pdfkeywords={Tuza's conjecture, triangle packing, triangle covering, bounded-degree graphs},
            pdfdisplaydoctitle=true}

\input{glyphtounicode}
\microtypesetup{protrusion=true,expansion=false}

\theoremstyle{plain}
\newtheorem{theorem}{Theorem}[section]
\newtheorem{lemma}[theorem]{Lemma}
\newtheorem{proposition}[theorem]{Proposition}
\newtheorem{corollary}[theorem]{Corollary}
\newtheorem{conjecture}[theorem]{Conjecture}
\theoremstyle{definition}
\newtheorem{definition}[theorem]{Definition}
\newtheorem{remark}[theorem]{Remark}
\newtheorem{observation}[theorem]{Observation}
\newtheorem{question}[theorem]{Question}

\newcommand{\mad}{\operatorname{mad}}
\newcommand{\WKE}{\mathrm{WKE}}
\newcommand{\cod}{\operatorname{c}}
\newcommand{\N}{N}
\newcommand{\dg}{\operatorname{d}}
\newcommand{\vc}{\beta}

\title[Tuza's conjecture for maximum degree seven]
  {Tuza's conjecture for graphs of maximum degree at most seven}
\author{Anish Gupta}
\address{Independent researcher}
\email{ag2269@cantab.ac.uk}
\urladdr{https://orcid.org/0009-0008-8137-7729}
\date{6 August 2026}

\begin{document}

\begin{abstract}
Tuza conjectured that every finite simple graph $G$ satisfies
$\tau(G)\le 2\nu(G)$, where $\nu(G)$ is the maximum number of pairwise
edge-disjoint triangles and $\tau(G)$ is the minimum number of edges
whose deletion makes $G$ triangle-free. Puleo proved the conjecture for every
graph of maximum average degree less than~$7$; this covers maximum degree at
most~$6$ but no $7$-regular graph. We prove the conjecture for maximum degree
at most~$7$.

The proof uses Puleo's reducible-set framework. At average degree seven his
discharging step no longer forces a reducible configuration. In a minimal
$7$-regular counterexample every vertex link is a connected seven-vertex graph
outside the weak K\H{o}nig--Egerv\'ary class. An exhaustive census of such
links supplies, at every vertex, an incident edge lying in four, five or six
triangles. We prove that its
endpoints form a reducible pair: codegrees five and six use a
packing-and-cover template and Fano-plane witnesses, while codegree four uses
an explicit catalogue of $1{,}144$ machine-checked local certificates. We
do not provide a human-readable proof of that catalogue; the certificates and
their verifiers accompany the paper. The constant~$2$ is sharp already at
maximum degree three.
\end{abstract}

\maketitle

\section{Introduction}

All graphs in this paper are finite, simple and undirected. A
\emph{triangle packing} of a graph $G$ is a set of triangles of $G$ that are
pairwise \emph{edge}-disjoint; distinct triangles of a packing may share
vertices. We write $\nu(G)$ for the largest size of a triangle packing. A
\emph{triangle transversal} (or \emph{triangle edge cover}) is a set
$Y\subseteq E(G)$ such that $G-Y$ is triangle-free, and $\tau(G)$ is the
smallest size of one. Every transversal meets each member of an
edge-disjoint packing, and those meetings use distinct edges, so
$\nu(G)\le\tau(G)$. Conversely, deleting all three edges of every triangle
in a maximum packing leaves no triangle: a surviving triangle would be
edge-disjoint from the packing and could be added to it. Thus
$\tau(G)\le 3\nu(G)$.

\begin{conjecture}[Tuza \cite{Tuza1981,Tuza1990}]
Every graph $G$ satisfies $\tau(G)\le 2\nu(G)$.
\end{conjecture}

Tuza's original formulation gives the conjecture a second interpretation
\cite{Tuza1981}. Let $\alpha_1(G)$ be the maximum number of edges in a
triangle-free subgraph of $G$, and let $\rho_0(G)$ be the minimum number of
pairwise edge-disjoint triangles and single edges whose union is $E(G)$.
Then $\alpha_1(G)=|E(G)|-\tau(G)$ and
$\rho_0(G)=|E(G)|-2\nu(G)$, so $\rho_0(G)\le\alpha_1(G)$ is equivalent to
Tuza's conjecture.

The conjecture is open. Haxell proved the universal bound
$\tau(G)\le (66/23)\,\nu(G)$ \cite{Haxell1999}. Among the classes for which the
conjecture is known are planar graphs (Tuza \cite{Tuza1990});
triangle-$3$-colourable graphs, a class containing all $4$-colourable graphs
(Aparna Lakshmanan, Bujt\'as and Tuza \cite{AparnaBujtasTuza2012}); and
$K_5$-free chordal graphs (Tuza \cite[Proposition 3(b)]{Tuza1990}).
Krivelevich proved the conjecture for graphs with no subdivision of $K_{3,3}$
\cite{Krivelevich1995}. Botler,
Fernandes and Guti\'errez strengthened the last result to $K_8$-free chordal
graphs by proving the conjecture for graphs of treewidth at most six; they
also proved $\tau\le\tfrac32\nu$ for every planar triangulation other than
$K_4$ \cite{BotlerFG2021}.
Kahn and Park proved that for every edge-probability function $p=p(n)$, the
Erd\H{o}s--R\'enyi random graph $G(n,p)$ satisfies the conjecture
asymptotically almost surely \cite{KahnPark2022}. Most relevant here is the
maximum-average-degree result. Write
\[
  \mad(G)\;=\;\max\Bigl\{\tfrac{2|E(H)|}{|V(H)|}\;:\;
    H\subseteq G,\ V(H)\neq\emptyset\Bigr\}
\]
for the maximum average degree of $G$.
For the empty graph we set $\mad(\emptyset)=0$.

\begin{theorem}[Puleo {\cite[Theorem 1.2]{Puleo2015}}]\label{thm:puleo}
If $\mad(G)<7$, then $\tau(G)\le 2\nu(G)$.
\end{theorem}

The hypothesis of Theorem~\ref{thm:puleo} is strict, and the strictness has a
sharp consequence for bounded-degree graphs. If $\Delta(G)\le 6$ then every
subgraph has average degree at most $6$, so $\mad(G)\le 6<7$ and
Theorem~\ref{thm:puleo} applies. If, on the other hand, $G$ is $7$-regular,
then $\mad(G)=7$ and Theorem~\ref{thm:puleo} says nothing at all about $G$.
The $7$-regular graphs are therefore exactly what separates ``maximum degree
at most six'', which is covered, from ``maximum degree at most seven'', which
is not. Our main theorem closes that gap.

There is an apparent one-edge shortcut, but it misses the conjectured bound by
exactly the amount that matters. If $G$ is connected and $7$-regular and
$e\in E(G)$, then $G-e$ is a proper subgraph, so
$\mad(G-e)<7$ by Lemma~\ref{lem:madhinge}(b). Puleo's theorem gives
$\tau(G-e)\le2\nu(G-e)$, while
$\tau(G)\le\tau(G-e)+1$ and $\nu(G-e)\le\nu(G)$. Thus this shortcut yields
only $\tau(G)\le2\nu(G)+1$. Recovering that last unit is the point of the
reducible-pair argument below.

\begin{theorem}\label{thm:main}
Every graph $G$ with $\Delta(G)\le 7$ satisfies $\tau(G)\le 2\nu(G)$.
\end{theorem}

Theorem~\ref{thm:main} assumes only the degree bound; $G$ may be disconnected
and its order is unrestricted. It is deduced in Section~\ref{sec:regular}
from the $7$-regular case.

\begin{theorem}\label{thm:regular}
Every $7$-regular graph $G$ satisfies $\tau(G)\le 2\nu(G)$.
\end{theorem}

The constant $2$ cannot be improved inside the class of
Theorem~\ref{thm:main}; the extremal examples are already present at maximum
degree three.

\begin{proposition}\label{prop:sharpintro}
$K_4$ has $(\nu,\tau)=(1,2)$ and $\Delta=3$; $K_5$ has $(\nu,\tau)=(2,4)$;
and for every $k\ge 1$ there are connected graphs $G_6,G_7$ with
$\Delta(G_6)\le 6$, $\Delta(G_7)=7$, and
$\nu(G_6)=\nu(G_7)=k$, $\tau(G_6)=\tau(G_7)=2k$. Consequently no
inequality $\tau\le c\,\nu$ with $c<2$, and no strict inequality
$\tau<2\nu$, holds for all graphs of maximum degree at most $7$.
\end{proposition}

Proposition~\ref{prop:sharpintro} is proved in Section~\ref{sec:sharp}.

\subsection{The obstruction at degree seven}\label{subsec:obstruction}

Puleo's proof of Theorem~\ref{thm:puleo} is a discharging argument built on a
notion of \emph{reducible set} (Definition~\ref{def:reducible} below): a
vertex set $V_0$ is reducible if a bounded amount of local packing and
covering data lets one delete $V_0$ together with the edges it interferes with
and then recurse. Consequently a minimal counterexample has no reducible set;
our task is to force one from the local structure at degree seven.

Puleo's singleton reduction itself has no degree hypothesis. His
Lemma~4.4 \cite{Puleo2015}, restated as Lemma~\ref{lem:puleo44}, makes
$\{v\}$ reducible whenever the link $G[\N(v)]$ is weak
K\H{o}nig--Egerv\'ary ($\WKE$). What fails at average degree seven is the discharging step that
forces a reducible configuration from the strict inequality
$\mad(G)<7$ \cite[Lemmas 2.7, 2.8]{Puleo2015}. Reducible pairs also appear in
the literature. Botler, Fernandes and Guti\'errez
\cite[Lemma 3.3]{BotlerFG2021} show that, in an irreducible robust graph, a
pair $\{x,y\}$ with $\dg(x),\dg(y)\le6$ and
$|\N(x)\cup\N(y)|\le7$ can survive only in a specific configuration: both
degrees are five, the common neighbourhood has order three, and both links
are $K_5$. Such a pair is necessarily non-adjacent
\cite[Lemma 3.2(b)]{BotlerFG2021}. Their configuration therefore differs in
two ways from ours: its vertices have degree at most six, whereas ours have
degree seven, and they reduce a non-edge, whereas we reduce the endpoints of
an edge.

A hypothetical counterexample nevertheless has strong links: each is a
connected seven-vertex graph outside $\WKE$
(Lemma~\ref{lem:linkstructure}). That condition alone does not make the
central vertex reducible. In $K_8$ every link is $K_7$, hence connected and
outside $\WKE$, but no singleton of $K_8$ is reducible
(Proposition~\ref{prop:k8}). Thus the link condition derived from Puleo's
singleton lemma cannot by itself yield another singleton reduction. We use it
instead to find an edge of high triangle codegree and reduce the two degree-$7$
endpoints of that edge.

\subsection{Outline of the proof}\label{subsec:method}

We reduce at an edge of high triangle codegree, where the
\emph{triangle codegree} $\cod(uv)=|\N(u)\cap\N(v)|$ of an edge is the number
of triangles containing it. Suppose
Theorem~\ref{thm:regular} fails and let $G$ be a connected $7$-regular graph
with $\tau(G)>2\nu(G)$. A short argument about maximum average degree
(Lemma~\ref{lem:madhinge}) shows that every \emph{proper} subgraph of $G$
falls under Theorem~\ref{thm:puleo}, so $G$ is minimal in the sense the
framework needs and has no reducible set at all. Then:

\begin{enumerate}
  \item every link $G[\N(v)]$ is connected, has seven vertices, and is
    \emph{not} in $\WKE$ (Lemma~\ref{lem:linkstructure});
  \item every connected non-$\WKE$ graph on seven vertices has at least three
    vertices of degree at least four (Lemma~\ref{lem:B1}), so some edge $uv$
    of $G$ satisfies $\cod(uv)\ge 4$, while $7$-regularity forces
    $\cod(uv)\le 6$ (Corollary~\ref{cor:trichotomy});
  \item for every edge $uv$ with $\cod(uv)\in\{4,5,6\}$ arising in such a $G$,
    the \emph{pair} $\{u,v\}$ is reducible
    (Propositions~\ref{prop:codeg6}, \ref{prop:codeg5} and
    \ref{prop:codeg4}).
\end{enumerate}

Step (3) contradicts the absence of reducible sets, and
Theorem~\ref{thm:regular} follows.

The last step is finite because a codegree-$c$ edge in a $7$-regular graph has
a fixed neighbourhood pattern. Writing $C=\N(u)\cap\N(v)$ and letting $A,B$
be the exclusive neighbours of $u,v$, one has $|C|=c$ and
$|A|=|B|=6-c$ (Figure~\ref{fig:local}). All triangles through either endpoint
lie in the resulting local graph, so a local certificate transfers to $G$.
A packing-and-cover template settles codegrees six and five apart from
complete cores, which admit Fano-plane witnesses. At codegree four, $1{,}144$
admissible local isomorphism classes remain and are handled by explicit
certificates.

These are among four finite statements verified exhaustively: the
seven-vertex link census, the two template classifications (including the
codegree-five exceptional list), and the codegree-four catalogue. Everything
else is proved in the text. Appendix~\ref{app:programs} describes the checkers
and the precise proof/computation boundary.

\subsection{Notation}

$\N(v)$ is the open neighbourhood of $v$ and $\N[v]=\N(v)\cup\{v\}$;
$\dg(v)=|\N(v)|$. For $W\subseteq V(G)$, $G[W]$ is the induced subgraph. The
\emph{link} of $v$ is $G[\N(v)]$. For an edge $uv\in E(G)$ we write
\[
  \cod(uv)\;=\;|\N(u)\cap \N(v)|
\]
for its \emph{triangle codegree}, the number of triangles of $G$ containing
$uv$. For an edge set $X$, $G-X$ is $G$ with the edges of $X$ deleted; for a
vertex set $V_0$, $G-V_0$ is $G$ with the vertices of $V_0$ and all incident
edges deleted. $\alpha(H)$ and $\vc(H)$ denote the independence number and the
vertex cover number of $H$, so that $\alpha(H)+\vc(H)=|V(H)|$. A subgraph
$J\subseteq G$ is \emph{proper} if $(V(J),E(J))\neq(V(G),E(G))$.

\section{Puleo's reducibility framework}\label{sec:prelim}

We recall the notions and results of \cite{Puleo2015} that we use, with their
exact hypotheses. They are used as a black box; in particular we do not use
Puleo's discharging lemmas (\cite[Lemmas 2.7, 2.8]{Puleo2015}), whose strict
average-degree forcing argument does not supply a configuration at equality.

\begin{definition}[{\cite[Definition 2.1]{Puleo2015}}]\label{def:reducible}
When $S$ is a set of triangles, an \emph{$S$-edge} is an edge of some triangle
in $S$. A nonempty set $V_0\subseteq V(G)$ is \emph{reducible} if there exist
a set $S$ of edge-disjoint triangles in $G$ and a set $X$ of edges of $G$ such
that:
\begin{enumerate}
  \item[(i)] $|X|\le 2|S|$;
  \item[(ii)] $G-X$ has no triangle containing a vertex of $V_0$; and
  \item[(iii)] $X$ contains every $S$-edge whose endpoints are both outside
    $V_0$.
\end{enumerate}
We then say that $V_0$ is reducible \emph{using $S$ and $X$}.
\end{definition}

Condition (iii) is easy to overlook but essential: it prevents the recursion
from re-using an edge already spent by $S$.

\begin{lemma}[{\cite[Lemma 2.2]{Puleo2015}}]\label{lem:puleo22}
Let $G$ be a graph and let $V_0\subseteq V(G)$ be reducible using $S$ and $X$.
Let $G'=(G-X)-V_0$. If $\tau(G')\le 2\nu(G')$, then $\tau(G)\le 2\nu(G)$.
\end{lemma}

\begin{definition}[{\cite[Definition 4.2]{Puleo2015}}]\label{def:redgeset}
A nonempty edge set $E_0\subseteq E(G)$ is \emph{reducible} if there exist a
set $S$ of edge-disjoint triangles and a set $X$ of edges of $G$ such that
(i) $|X|\le 2|S|$; (ii) $G-X$ has no triangle containing an edge of $E_0$;
and (iii) $X$ contains every $S$-edge that is not in $E_0$.
\end{definition}

\begin{lemma}[{\cite[Lemma 4.3]{Puleo2015}}]\label{lem:puleo43}
Let $G$ be a graph and let $E_0\subseteq E(G)$ be reducible using $S$ and $X$.
Let $G'=(G-X)-E_0$. If $\tau(G')\le 2\nu(G')$, then $\tau(G)\le 2\nu(G)$.
\end{lemma}

\begin{definition}[{\cite[Definition 4.1]{Puleo2015}}]\label{def:wke}
A graph $H$ is a \emph{weak K\H{o}nig--Egerv\'ary} graph if $H$ has a matching
$M$ and a vertex set $Q\subseteq V(H)$ such that $|Q|\le|M|$ and $Q$ is a
vertex cover in $H-M$. We write $\WKE$ for the class of such graphs, and say
that the pair $(M,Q)$ \emph{witnesses} $H\in\WKE$.
\end{definition}

Note that $M$ and $Q$ range over \emph{all} matchings and \emph{all} vertex
sets of the stated size; in particular $M=\emptyset$, $Q=\emptyset$ witnesses
that every edgeless graph lies in $\WKE$. Every K\H{o}nig--Egerv\'ary graph,
hence every bipartite graph, lies in $\WKE$.

\begin{lemma}[{\cite[Lemma 4.4]{Puleo2015}}]\label{lem:puleo44}
Let $v\in V(G)$ and let $G_0$ be a nonempty union of components of
$G[\N(v)]$. If $G_0\in\WKE$, then $G$ has a reducible set of edges. Also, if
$G[\N(v)]\in\WKE$, then $\{v\}$ is reducible.
\end{lemma}

For the first claim, Puleo's proof exhibits the reducible edge set explicitly
as $E_0=\{vw : w\in V(G_0)\}$.

\begin{corollary}[{\cite[Corollary 4.9]{Puleo2015}}]\label{cor:puleo49}
If $H$ is connected and $|V(H)|\le 4$, then $H\in\WKE$.
\end{corollary}

\begin{definition}[{\cite[Definition 2.3]{Puleo2015}}]\label{def:robust}
A graph $G$ is \emph{robust} if for every $v\in V(G)$, every component of
$G[\N(v)]$ has order at least $5$.
\end{definition}

\section{Reduction to the \texorpdfstring{$7$}{7}-regular case}\label{sec:regular}

The next lemma identifies the equality case in Puleo's strict hypothesis.
Within maximum degree seven, a connected graph either falls under
Theorem~\ref{thm:puleo} outright, or is $7$-regular and has \emph{all} of its
proper subgraphs falling under it.

\begin{lemma}\label{lem:madhinge}
Let $K$ be a connected graph with $\Delta(K)\le 7$.
\begin{enumerate}
  \item[(a)] If $K$ is not $7$-regular, then $\mad(K)<7$.
  \item[(b)] If $K$ is $7$-regular, then every proper subgraph $J\subsetneq K$
    satisfies $\mad(J)<7$.
\end{enumerate}
\end{lemma}

\begin{proof}
Both parts follow from one computation. Suppose $L$ is a nonempty subgraph of
$K$ with $2|E(L)|/|V(L)|\ge 7$, i.e.\ with average degree at least $7$. Every
vertex of $L$ has degree at most $\Delta(K)\le 7$ in $L$, so the average being
at least $7$ forces every vertex of $L$ to have degree exactly $7$ in $L$.
Fix $x\in V(L)$. Then $x$ has $7$ neighbours inside $L$; since $x$ also has at
most $7$ neighbours in $K$, all neighbours of $x$ in $K$ lie in $V(L)$ and all
edges of $K$ at $x$ lie in $E(L)$. As $x$ was arbitrary, $V(L)$ is closed
under $K$-adjacency and $E(L)$ contains every $K$-edge incident with $V(L)$.
Since $K$ is connected and $V(L)\neq\emptyset$, we get $V(L)=V(K)$ and
$E(L)=E(K)$, that is, $L=K$.

For (a): if $K$ is not $7$-regular then $L=K$ is impossible, so no such $L$
exists and $\mad(K)<7$. For (b): if $J\subsetneq K$ is proper and
$\mad(J)\ge 7$, take such an $L\subseteq J$; then $L=K$, so $J\supseteq K$,
contradicting properness. (If $E(K)=\emptyset$ the statement is vacuous, since
$\mad$ of an edgeless graph is $0$.)
\end{proof}

\begin{proposition}\label{prop:AtoAprime}
Theorem~\ref{thm:regular} implies Theorem~\ref{thm:main}.
\end{proposition}

\begin{proof}
Both $\nu$ and $\tau$ are additive over connected components: every triangle
lies in a single component, the edge sets of distinct components are disjoint,
and a union of packings (respectively transversals) over the components is a
packing (respectively transversal) of the whole graph, optimal if each part
is. So it suffices to prove $\tau(K)\le 2\nu(K)$ for every connected $K$ with
$\Delta(K)\le 7$.

If $K$ is $7$-regular, this is Theorem~\ref{thm:regular}. If not, then
$\mad(K)<7$ by Lemma~\ref{lem:madhinge}(a), and Theorem~\ref{thm:puleo}
applies.
\end{proof}

\section{The structure of a minimal \texorpdfstring{$7$}{7}-regular counterexample}
\label{sec:minimal}

Throughout this section and Sections~\ref{sec:local}--\ref{sec:codeg4} we fix
the following.

\begin{quote}
\textbf{Standing hypothesis.} $G$ is a connected $7$-regular graph with
$\tau(G)>2\nu(G)$.
\end{quote}

Such a $G$ exists if Theorem~\ref{thm:regular} fails: a counterexample has a
component that is a counterexample, by additivity, and every component of a
$7$-regular graph is $7$-regular and connected. Note $|V(G)|\ge 8$ and
$|V(G)|$ is even.

\begin{lemma}\label{lem:proper}
Every proper subgraph of $G$ satisfies Tuza's conjecture.
\end{lemma}

\begin{proof}
Immediate from Lemma~\ref{lem:madhinge}(b) and Theorem~\ref{thm:puleo}.
\end{proof}

\begin{lemma}\label{lem:norreducible}
$G$ has no reducible vertex set and no reducible edge set.
\end{lemma}

\begin{proof}
Suppose $V_0\subseteq V(G)$ is nonempty and reducible using $S$ and $X$, and
put $G'=(G-X)-V_0$. Since $V_0\ne\emptyset$, $V(G')\subsetneq V(G)$, so $G'$
is a proper subgraph of $G$ and $\tau(G')\le 2\nu(G')$ by
Lemma~\ref{lem:proper}. Lemma~\ref{lem:puleo22} then gives
$\tau(G)\le 2\nu(G)$, contradicting the standing hypothesis.

Likewise, if $E_0\subseteq E(G)$ is nonempty and reducible using $S$ and $X$,
then $G'=(G-X)-E_0$ has $E(G')\subsetneq E(G)$, so $G'$ is proper,
$\tau(G')\le2\nu(G')$, and Lemma~\ref{lem:puleo43} gives the same
contradiction.
\end{proof}

\begin{lemma}\label{lem:robust}
$G$ is robust.
\end{lemma}

\begin{proof}
Let $v\in V(G)$ and let $G_0$ be a component of $G[\N(v)]$ with
$|V(G_0)|\le 4$. Then $G_0$ is connected on at most four vertices, so
$G_0\in\WKE$ by Corollary~\ref{cor:puleo49}. By Lemma~\ref{lem:puleo44} the
edge set $E_0=\{vw : w\in V(G_0)\}$ is reducible; it is nonempty because
$G_0$ is a component of the nonempty graph $G[\N(v)]$ (recall $\dg(v)=7$).
This contradicts Lemma~\ref{lem:norreducible}.
\end{proof}

Lemma~\ref{lem:robust} is \cite[Lemma 2.6]{Puleo2015}; we have repeated its
proof to make explicit that the form of minimality it needs is the one
supplied by Lemma~\ref{lem:proper}.

\begin{lemma}\label{lem:linkstructure}
For every $v\in V(G)$, the link $H_v=G[\N(v)]$ is a connected graph on seven
vertices with $H_v\notin\WKE$.
\end{lemma}

\begin{proof}
$|V(H_v)|=\dg(v)=7$. By Lemma~\ref{lem:robust} every component of $H_v$ has at
least $5$ vertices; two components would need at least $10>7$ vertices, so
$H_v$ is connected. If $H_v\in\WKE$ then $\{v\}$ is reducible by
Lemma~\ref{lem:puleo44}, contradicting Lemma~\ref{lem:norreducible}.
\end{proof}

Lemma~\ref{lem:linkstructure} extracts what the singleton reduction can give:
the negative information $H_v\notin\WKE$. The following finite fact turns that
information into positive local structure. A vertex of degree $k$ in the link
$H_v$ is an edge of triangle codegree $k$ at $v$.

\begin{lemma}[verified exhaustively]\label{lem:B1}
Every connected graph $H$ on seven vertices with $H\notin\WKE$ has at least
three vertices of degree at least four.
\end{lemma}

\emph{Verification.} There are $2^{21}=2{,}097{,}152$ labelled graphs on a
fixed set of seven vertices. The program \texttt{tests/check\_codegree4.py}
enumerates all of them, decides membership in $\WKE$ directly from
Definition~\ref{def:wke} by enumerating all matching/cover witness pairs,
decides connectivity directly, and reports the degree profile. It finds
$167{,}871$ non-$\WKE$ graphs, of which $166{,}793$ are
connected, and confirms that every one of the latter has at least three
vertices of degree at least four. The bound is attained: exactly $4{,}620$ of
them have exactly three such vertices. \hfill$\square$

\begin{remark}[connectivity is essential]\label{rem:connectivity}
Lemma~\ref{lem:B1} is false without the connectivity hypothesis: the same
enumeration finds $315$ labelled non-$\WKE$ graphs on seven vertices with
fewer than three vertices of degree at least four, and every one of them is
disconnected. This is why Lemma~\ref{lem:linkstructure} must establish
connectivity of the link, and hence why robustness (Lemma~\ref{lem:robust})
is needed before Lemma~\ref{lem:B1} can be applied.
\end{remark}

\begin{corollary}\label{cor:trichotomy}
$G$ has an edge $uv$ with $\cod(uv)\in\{4,5,6\}$. In fact every vertex of $G$
is incident with at least three such edges.
\end{corollary}

\begin{proof}
Fix $v\in V(G)$ and let $H_v=G[\N(v)]$. For $u\in \N(v)$ the degree of $u$
inside $H_v$ is $|\N(u)\cap \N(v)|=\cod(uv)$. By
Lemma~\ref{lem:linkstructure}, $H_v$ is connected, on seven vertices and not
in $\WKE$, so by Lemma~\ref{lem:B1} at least three of its vertices have degree
at least four; that is, at least three edges $uv$ at $v$ have $\cod(uv)\ge4$.
For the upper bound, if $uv\in E(G)$ then
$\N(u)\cap \N(v)\subseteq \N(u)\setminus\{v\}$, which has $6$ elements, so
$\cod(uv)\le 6$.
\end{proof}

Only the existence of one edge of codegree at least four is used below.

\subsection{Why the link condition does not reduce a singleton}
\label{subsec:k8}

Before turning to the reductions themselves we record the fact, announced in
Section~\ref{subsec:obstruction}, that the conclusion of
Lemma~\ref{lem:linkstructure} cannot be used to reduce a single vertex. The
relevant example is $K_8$. It is $7$-regular, and it satisfies Tuza's
conjecture ($\nu(K_8)=8$, $\tau(K_8)=12$), so it is not a counterexample to
anything; the point is that it has the same connected non-$\WKE$ links and
yet no singleton in it is reducible.

\begin{observation}\label{obs:k7nonwke}
$K_7\notin\WKE$, so every link of $K_8$ is connected and not in $\WKE$, and
$K_8$ is robust.
\end{observation}

\begin{proof}
A matching $M$ of $K_7$ has $|M|\le3$. The graph $K_7-M$ has independence
number at most $2$: three pairwise non-adjacent vertices of $K_7-M$ would be
pairwise joined by edges of $M$, i.e.\ would form a triangle inside $M$, which
is impossible for a matching. Hence
$\vc(K_7-M)=7-\alpha(K_7-M)\ge 5>3\ge|M|$, so no pair $(M,Q)$ witnesses
$K_7\in\WKE$. Each link of $K_8$ is $K_7$, which is connected on seven
vertices, so every component of every link has order $7\ge5$.
\end{proof}

\begin{proposition}\label{prop:k8}
No singleton $\{v\}\subseteq V(K_8)$ is reducible.
\end{proposition}

\begin{proof}
Write $G=K_8$, fix $v$, and let $U=V(G)\setminus\{v\}$, so $|U|=7$ and
$G[U]=K_7$. Suppose $\{v\}$ is reducible using $S$ and $X$.

Let $D=\{x\in U : vx\in X\}$ and $d=|D|$. Since $G-X$ has no triangle
containing $v$, for all distinct $x,y\in U\setminus D$ the triangle $vxy$
must meet $X$, and since $vx,vy\notin X$ this forces $xy\in X$. Hence
\[
  F:=\binom{U\setminus D}{2}\subseteq X,\qquad |F|=\tbinom{7-d}{2},
\]
where $\binom{Z}{2}$ denotes the set of all unordered pairs from $Z$.

Split $S$ into the $k$ triangles containing $v$ and the $m$ triangles avoiding
$v$, so $|S|=k+m$. A triangle of $S$ through $v$ is $vxy$ with $x,y\in U$; two
such triangles are edge-disjoint only if their vertex pairs are disjoint, so
those pairs form a matching in $U$ and $k\le3$. Let $P\subseteq\binom{U}{2}$
be the set of $S$-edges avoiding $v$: it consists of the $k$ ``rim'' edges
$xy$ of the triangles $vxy\in S$, together with the $3m$ edges of the $m$
triangles of $S$ inside $U$. All of these are distinct, because $S$ is
edge-disjoint, so $|P|=3m+k$; and by condition (iii) of
Definition~\ref{def:reducible}, $P\subseteq X$.

Since $X$ contains the $d$ edges $\{vx : x\in D\}$ and the set $F\cup P$ of
edges inside $U$, and these are disjoint families,
\begin{equation}\label{eq:k8}
  d+|F\cup P|\;\le\;|X|\;\le\;2|S|\;=\;2(k+m).
\end{equation}
Using $|F\cup P|\ge|P|=3m+k$ in \eqref{eq:k8} gives $d+3m+k\le 2k+2m$, i.e.
\begin{equation}\label{eq:k8a}
  d+m\;\le\;k\;\le\;3 .
\end{equation}
Using $|F\cup P|\ge|F|$ in \eqref{eq:k8} and $k\le3$ gives
\begin{equation}\label{eq:k8b}
  d+\tbinom{7-d}{2}\;\le\;2(k+m)\;\le\;6+2m .
\end{equation}
By \eqref{eq:k8a}, $d\le3$ and $m\le3-d$. We check the four cases:
\[
\begin{array}{c|c|c|c}
 d & d+\binom{7-d}{2} & \text{\eqref{eq:k8b} requires} & \text{\eqref{eq:k8a} allows}\\
\hline
 0 & 21 & m\ge 8 & m\le 3\\
 1 & 16 & m\ge 5 & m\le 2\\
 2 & 12 & m\ge 3 & m\le 1\\
 3 &  9 & m\ge 2 & m\le 0\\
\end{array}
\]
Every case is contradictory, so $\{v\}$ is not reducible.
\end{proof}

Thus the connected non-$\WKE$ link condition supplied by
Lemma~\ref{lem:linkstructure} does not imply singleton reducibility. The
reductions in Sections~\ref{sec:codeg6}--\ref{sec:codeg4} use the additional
high-codegree edge structure and act on \emph{pairs}.

\section{The local graph at an edge, and ambient safety}\label{sec:local}

The three reductions in Sections~\ref{sec:codeg6}--\ref{sec:codeg4} concern a
bounded neighbourhood of an edge, but they must hold inside an unknown ambient
graph $G$. This section proves that such local reasoning is legitimate.

Throughout, $G$ is $7$-regular (the standing hypothesis is not needed for this
section) and $uv\in E(G)$ with $c=\cod(uv)$. Put
\[
  C=\N(u)\cap \N(v),\qquad
  A=\N(u)\setminus(\{v\}\cup C),\qquad
  B=\N(v)\setminus(\{u\}\cup C).
\]

\begin{lemma}\label{lem:disjoint}
$\{u,v\}$, $A$, $B$ and $C$ are pairwise disjoint, $|C|=c$ and
$|A|=|B|=6-c$. In particular $W:=\{u,v\}\cup A\cup B\cup C$ satisfies
$|W|=2+c+2(6-c)=14-c$.
\end{lemma}

\begin{proof}
$A\cap C=B\cap C=\emptyset$ and $A\cap\{v\}=B\cap\{u\}=\emptyset$ by
definition, and $u\notin \N(u)$, $v\notin \N(v)$ since $G$ is simple. If
$x\in A\cap B$ then $x\in \N(u)$ and $x\in \N(v)$, so $x\in C$, a
contradiction; hence $A\cap B=\emptyset$. Also $u\notin B$ and $v\notin A$ by
definition, and $u,v\notin C$. Since $\dg(u)=7$ and
$\N(u)=\{v\}\sqcup C\sqcup A$, we get $|A|=6-c$; symmetrically $|B|=6-c$.
Thus $|W|=2+c+2(6-c)=14-c$.
\end{proof}

\begin{definition}\label{def:localgraph}
The \emph{local graph} $L=L(G,uv)$ is the graph with vertex set $W$ and edge
set consisting of all edges of $G[W]$ except those with one endpoint in $A$
and the other in $B$.
\end{definition}

Explicitly, $E(L)$ consists of $uv$; the edges $ux$ for $x\in C\cup A$; the
edges $vx$ for $x\in C\cup B$; and all $G$-edges inside $C\cup A$, inside
$C\cup B$, and inside $C$. There are no $G$-edges $vy$ with $y\in A$ or $ux$
with $x\in B$, by the definition of $A$ and $B$. See
Figure~\ref{fig:local}.

\begin{figure}[t]
\centering
\begin{tikzpicture}[
  blob/.style={draw,thick,shape=ellipse,align=center,inner sep=3pt,
               minimum width=30mm,minimum height=13mm},
  hub/.style={circle,fill=black,inner sep=1.7pt}]
  \node[blob] (C) at (0,2.5) {$C=\N(u)\cap\N(v)$\\[-3pt]{\footnotesize $|C|=c$}};
  \node[blob,minimum width=20mm] (A) at (-4.0,-2.1)
        {$A$\\[-3pt]{\footnotesize $|A|=6-c$}};
  \node[blob,minimum width=20mm] (B) at (4.0,-2.1)
        {$B$\\[-3pt]{\footnotesize $|B|=6-c$}};
  \node[hub] (u) at (-2.5,0) {};
  \node[hub] (v) at (2.5,0) {};
  \node at (-2.5,0.42) {$u$};
  \node at (2.5,0.42) {$v$};
  \draw[very thick] (u) -- (v);
  \node[font=\footnotesize] at (0,-0.3) {$uv$};
  \draw[very thick] (u) -- (C);
  \draw[very thick] (v) -- (C);
  \draw[very thick] (u) -- (A);
  \draw[very thick] (v) -- (B);
  \draw[dotted,thick] (A) -- (C);
  \draw[dotted,thick] (B) -- (C);
  \draw[dashed] (A) -- (B);
\end{tikzpicture}
\caption{The local graph $L=L(G,uv)$ at an edge of triangle codegree $c$ in a
$7$-regular graph. Thick lines denote complete adjacency: $u$ and $v$ are
adjacent to each other and to every vertex of $C$, $u$ to every vertex of $A$,
and $v$ to every vertex of $B$. Dotted lines denote adjacencies that are
arbitrary; so are the adjacencies inside $C$, inside $A$ and inside $B$. The
dashed $A$--$B$ adjacencies, together with all edges leaving $W$, are discarded
in passing from $G[W]$ to $L$; by Lemma~\ref{lem:localcomplete} they are
invisible to the triangles through $u$ or $v$.}
\label{fig:local}
\end{figure}
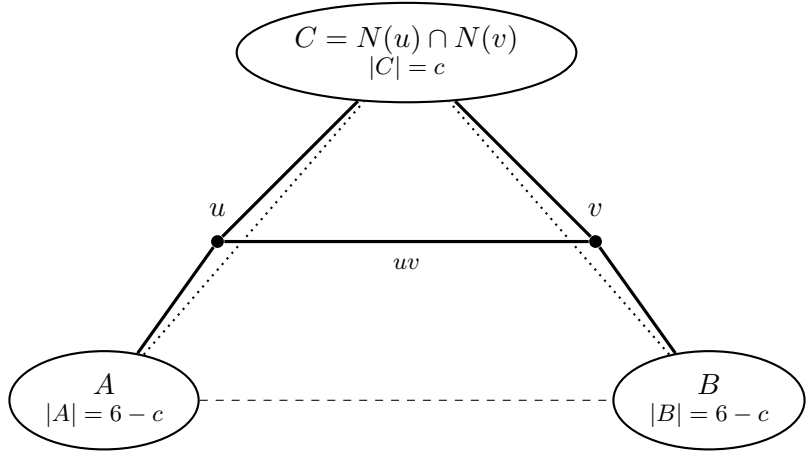

\begin{lemma}[completeness of the local model]\label{lem:localcomplete}
\begin{enumerate}
  \item[(a)] Every triangle of $G$ containing $u$ or $v$ has all three
    vertices in $W$ and all three edges in $E(L)$; hence it is a triangle
    of $L$.
  \item[(b)] Conversely, every triangle of $L$ is a triangle of $G$.
\end{enumerate}
\end{lemma}

\begin{proof}
(b) is immediate since $E(L)\subseteq E(G)$.

(a) Let $T=uxy$ be a triangle of $G$ through $u$ (the case of $v$ is
symmetric). Then $x,y\in \N(u)=\{v\}\cup C\cup A\subseteq W$, so all vertices
of $T$ lie in $W$. The edges $ux,uy$ have an endpoint $u\notin A\cup B$, so
neither is an $A$--$B$ edge, and both lie in $E(L)$. For the edge $xy$: if
$v\notin\{x,y\}$ then $x,y\in C\cup A$, so $xy$ is not an $A$--$B$ edge and
lies in $E(L)$; if, say, $x=v$, then $y\in \N(u)\cap \N(v)=C$ and $vy\in E(L)$.
\end{proof}

Lemma~\ref{lem:localcomplete}(a) says that $A$--$B$ edges, edges from $W$ to
$V(G)\setminus W$, and edges outside $W$ are all invisible to triangles
through $u$ or $v$. Therefore conditions checked in $L$ against triangles
through the hubs remain valid in the ambient graph.

\begin{corollary}[local reducibility criterion]\label{cor:localcriterion}
Let $S$ be a set of pairwise edge-disjoint triangles of $L$ and let
$X\subseteq E(L)$ satisfy
\begin{enumerate}
  \item[(i)] $|X|\le 2|S|$;
  \item[(ii)] every triangle of $L$ containing $u$ or $v$ contains an edge of
    $X$;
  \item[(iii)] every $S$-edge with both endpoints outside $\{u,v\}$ lies in
    $X$.
\end{enumerate}
Then $\{u,v\}$ is reducible in $G$, using $S$ and $X$.
\end{corollary}

\begin{proof}
By Lemma~\ref{lem:localcomplete}(b) the members of $S$ are triangles of $G$,
and they are pairwise edge-disjoint; also $X\subseteq E(L)\subseteq E(G)$.
Condition (i) of Definition~\ref{def:reducible} is (i), and condition (iii)
is (iii), verbatim with $V_0=\{u,v\}$. For condition (ii), suppose $G-X$ has
a triangle $T$ containing $u$ or $v$. By Lemma~\ref{lem:localcomplete}(a),
$T$ is a triangle of $L$ containing $u$ or $v$, so by (ii) it contains an edge
of $X$ --- contradicting $T\subseteq G-X$.
\end{proof}

For codegree four we also use the following arithmetic constraint imposed by
$7$-regularity on the common neighbours.

\begin{lemma}[degree budget]\label{lem:budget}
For every $x\in C$,
\[
  \dg_{G[C]}(x)+|\N(x)\cap A|+|\N(x)\cap B|\;\le\;5 .
\]
\end{lemma}

\begin{proof}
$x$ is adjacent to both $u$ and $v$, and $\dg_G(x)=7$, so $x$ has at most $5$
neighbours in $V(G)\setminus\{u,v\}$. The three quantities on the left count
disjoint sets of such neighbours, by Lemma~\ref{lem:disjoint}.
\end{proof}

\section{A packing and cover template}\label{sec:template}

For triangle codegrees five and six we use a uniform recipe for producing
witnesses, driven by a single set $R$ of edges inside $C$: the packing $S$ is
built from triangles supported by $R$, and the transversal $X$ consists of $R$
itself, the spokes at a vertex cover of the rest of $G[C]$, and a few forced
edges at the hubs. Everything then reduces to one inequality. Put $H=G[C]$ and,
for $R\subseteq E(H)$, define the following.

\begin{definition}\label{def:template}
\phantom{.}
\begin{itemize}
  \item $q(R)=\vc(H-R)$, the vertex cover number of the graph on $C$ with edge
    set $E(H)\setminus R$;
  \item $b(R)$ is the largest number of edges of a subgraph $F\subseteq R$
    that is the edge-disjoint union of two matchings --- equivalently
    (see below) whose components are paths and even cycles;
  \item for $x\in C$, $R-x$ denotes the set of edges of $R$ not incident with
    $x$; and
  \item $p(R)=\max\bigl(b(R),\ 1+\max_{x\in C}b(R-x)\bigr)$.
\end{itemize}
\end{definition}

The stated equivalence in the definition of $b$ is standard: a graph is the
union of two matchings if and only if its chromatic index is at most $2$, if
and only if its maximum degree is at most $2$ and it has no odd cycle, if and
only if its components are paths and even cycles.

\begin{lemma}[what $p(R)$ counts]\label{lem:packing}
Let $\mathcal{T}(R)$ be the set of triangles of the forms
\[
  uvx\ (x\in C),\qquad uxy\ (xy\in R),\qquad vxy\ (xy\in R).
\]
Then the maximum size of a pairwise edge-disjoint subfamily of
$\mathcal{T}(R)$ is exactly $p(R)$.
\end{lemma}

\begin{proof}
\emph{$p(R)$ is attained.} First, let $F\subseteq R$ have $|F|=b(R)$ and be
the disjoint union of matchings $M_1,M_2$. Take
$S=\{uxy : xy\in M_1\}\cup\{vxy : xy\in M_2\}$. The $u$-spokes used are
$\{ux,uy : xy\in M_1\}$, pairwise distinct because $M_1$ is a matching;
similarly for the $v$-spokes; and the rim edges used are the edges of
$M_1\sqcup M_2=F$, pairwise distinct. A $u$-spoke is never a $v$-spoke and
never a rim edge. So $S$ is edge-disjoint and $|S|=|M_1|+|M_2|=b(R)$.
Second, fix $x\in C$ and let $F\subseteq R-x$ have $|F|=b(R-x)$ with parts
$M_1,M_2$ as before. Take $S=\{uvx\}\cup\{uyz:yz\in M_1\}\cup\{vyz:yz\in
M_2\}$. The triangle $uvx$ uses $uv,ux,vx$; no other member of $S$ uses $uv$,
and no other member uses $ux$ or $vx$ because no edge of $F$ is incident with
$x$. So $S$ is edge-disjoint of size $1+b(R-x)$.

\emph{$p(R)$ is an upper bound.} Let $S\subseteq\mathcal{T}(R)$ be
edge-disjoint. At most one member of $S$ is of the form $uvx$, since any two
share the edge $uv$. Let $S_u=\{xy\in R : uxy\in S\}$ and
$S_v=\{xy\in R : vxy\in S\}$. If $xy,xz\in S_u$ with $y\ne z$ then $uxy$ and
$uxz$ share $ux$; so $S_u$ is a matching, and likewise $S_v$. If
$xy\in S_u\cap S_v$ then $uxy$ and $vxy$ share the rim edge $xy$; so
$S_u\cap S_v=\emptyset$. Hence $S_u\sqcup S_v$ is a subgraph of $R$ that is
the disjoint union of two matchings, so $|S_u|+|S_v|\le b(R)$.

If no member of $S$ has the form $uvx$, then $|S|=|S_u|+|S_v|\le b(R)$.
Otherwise $uvx\in S$ for exactly one $x$, and no member $uyz$ or $vyz$ of $S$
can have $x\in\{y,z\}$ (it would share $ux$, respectively $vx$, with $uvx$);
so $S_u\sqcup S_v\subseteq R-x$ and
$|S|=1+|S_u|+|S_v|\le 1+b(R-x)\le 1+\max_{x'}b(R-x')$. Either way
$|S|\le p(R)$.
\end{proof}

\begin{lemma}[template]\label{lem:template}
Let $G$ be $7$-regular and $uv\in E(G)$ with $c=\cod(uv)\in\{5,6\}$; set
$t=3$ if $c=5$ and $t=1$ if $c=6$. If some $R\subseteq E(H)$ satisfies
\begin{equation}\label{eq:template}
  t+|R|+2q(R)\;\le\;2p(R),
\end{equation}
then $\{u,v\}$ is reducible in $G$.
\end{lemma}

\begin{proof}
Choose a minimum vertex cover $Q$ of $H-R$, so $|Q|=q(R)$, and by
Lemma~\ref{lem:packing} choose an edge-disjoint $S\subseteq\mathcal{T}(R)$
with $|S|=p(R)$. Put
\[
  X=\begin{cases}
    \{uv,\,ua,\,vb\}\cup R\cup\{ux,vx : x\in Q\}, & c=5,\ A=\{a\},\ B=\{b\},\\[2pt]
    \{uv\}\cup R\cup\{ux,vx : x\in Q\}, & c=6,\ A=B=\emptyset.
  \end{cases}
\]
All listed edges lie in $E(L)$, and they are pairwise distinct: $uv$ joins the
two hubs, $ua$ and $vb$ join a hub to an exclusive neighbour, the elements of
$R$ join two vertices of $C$, and the spokes $ux,vx$ join a hub to $C$. Hence
$|X|=t+|R|+2q(R)$, and \eqref{eq:template} gives condition (i) of
Corollary~\ref{cor:localcriterion}.

Condition (ii). By Lemma~\ref{lem:localcomplete} it suffices to consider
triangles of $L$ containing $u$ or $v$; take one containing $u$, the case of
$v$ being symmetric. Its other two vertices lie in
$\N_L(u)=\{v\}\cup C\cup A$. The possibilities are:
\begin{itemize}
  \item $uvx$ with $x\in C$: it contains $uv\in X$.
  \item $uxy$ with $x,y\in C$ and $xy\in E(H)$: if $xy\in R$ then $xy\in X$;
    otherwise $xy$ is an edge of $H-R$, so $Q$ covers it and some endpoint,
    say $x$, lies in $Q$, whence $ux\in X$.
  \item $uax$ with $a\in A$, $x\in C$ (only when $c=5$): it contains
    $ua\in X$.
  \item There is no triangle $uab'$ with $a\in A,b'\in B$, since
    $B\cap \N(u)=\emptyset$; and none with two vertices of $A$, since
    $|A|\le1$ here.
\end{itemize}
Condition (iii). The members of $S$ are of the forms $uvx$, $uxy$, $vxy$ with
$xy\in R$. The edges of $uvx$ are $uv,ux,vx$, all incident with $\{u,v\}$. The
only edge of $uxy$ (respectively $vxy$) not incident with $\{u,v\}$ is
$xy\in R\subseteq X$. So every $S$-edge with both endpoints outside $\{u,v\}$
lies in $X$.

Corollary~\ref{cor:localcriterion} now applies.
\end{proof}

The template uses only triangles that meet $\{u,v\}$. At a complete core that
is not enough: no $R$ works, and a packing using triangles inside $C$ must be
produced by hand.

\section{Triangle codegree six}\label{sec:codeg6}

Here $c=6$, so $A=B=\emptyset$, $W=\{u,v\}\cup C$ has $8$ vertices, and $L$ is
determined by the graph $H=G[C]$ on the six labelled vertices of $C$: $u$ and
$v$ are each adjacent to all of $C$ and to each other. There are $2^{15}$
possibilities for $H$, and the template disposes of all but one.

\begin{proposition}[verified exhaustively]\label{prop:template6}
For every graph $H$ on six labelled vertices other than the complete graph
$K_6$, some $R\subseteq E(H)$ satisfies $1+|R|+2q(R)\le 2p(R)$. For $H=K_6$
no such $R$ exists.
\end{proposition}

\emph{Verification.} The program \texttt{tests/check\_codegree6.py} enumerates
all $2^{15}=32{,}768$ labelled graphs $H$ on $C$ and, for each, all subsets
$R\subseteq E(H)$, computing $q(R)$ by minimising over all $2^6$ vertex
subsets and $b(R)$ by maximising over all pairs of edge-disjoint matchings of
$K_6$. It reports an explicit witness $R$ for each of the $32{,}767$
non-complete $H$, and confirms that $K_6$ admits none among all $2^{15}$
choices of $R$. For exactly $155$ of the non-complete cores the best possible
choice of $R$ meets the inequality with equality; every other core admits
positive slack.
\hfill$\square$

The complete core must therefore be handled by a witness that uses triangles
inside $C$, which the template cannot supply. Such a witness exists, and it is
a Fano plane.

\begin{lemma}[a Fano witness for $K_6$]\label{lem:fano6}
Let $G$ be $7$-regular and $uv\in E(G)$ with $\cod(uv)=6$ and $G[C]\cong K_6$.
Then $\{u,v\}$ is reducible.
\end{lemma}

\begin{proof}
Write $C=\{c_0,c_1,c_2,c_3,c_4,z\}$. Since $G[C]=K_6$ and $u,v$ are adjacent
to everything in $C$ and to each other, $L$ is the complete graph on
$W=\{u,v\}\cup C$, i.e.\ $L\cong K_8$. Take
\[
  S=\{\,uvc_0,\ uc_1c_2,\ uc_3c_4,\ vc_1c_3,\ vc_2c_4,\ c_0c_1c_4,\
        c_0c_2c_3\,\},
\]
\[
  X=\{c_ic_j : 0\le i<j\le 4\}\cup\{uv,\ uz,\ vz\}.
\]
The seven triangles of $S$ are the lines of a Fano plane on the seven points
$\{u,v,c_0,\dots,c_4\}$: they use the $21$ edges
\[
  \begin{array}{l}
  uv,uc_0,vc_0;\quad uc_1,uc_2,c_1c_2;\quad uc_3,uc_4,c_3c_4;\quad
  vc_1,vc_3,c_1c_3;\\
  vc_2,vc_4,c_2c_4;\quad c_0c_1,c_1c_4,c_0c_4;\quad c_0c_2,c_2c_3,c_0c_3,
  \end{array}
\]
which are pairwise distinct; so $S$ is edge-disjoint and $|S|=7$. Also
$|X|=10+3=13\le 14=2|S|$.

Condition (ii) of Corollary~\ref{cor:localcriterion}: a triangle of $L$
containing $u$ is $uxy$ with $x,y\in \N_L(u)=\{v\}\cup C$. If $v\in\{x,y\}$ it
contains $uv\in X$. If $x,y\in\{c_0,\dots,c_4\}$ it contains
$xy\in X$. Otherwise $z\in\{x,y\}$ and the triangle contains $uz\in X$.
Symmetrically for $v$, using $vz\in X$.

Condition (iii): the $S$-edges with both endpoints outside $\{u,v\}$ are
$c_1c_2$, $c_3c_4$, $c_1c_3$, $c_2c_4$, $c_0c_1$, $c_1c_4$, $c_0c_4$,
$c_0c_2$, $c_2c_3$, $c_0c_3$, all of which lie in
$\{c_ic_j: 0\le i<j\le 4\}\subseteq X$.
\end{proof}

Note that two of the seven triangles of $S$, namely $c_0c_1c_4$ and
$c_0c_2c_3$, lie wholly inside $C$ and meet neither hub; by
Proposition~\ref{prop:template6} the template of
Section~\ref{sec:template} cannot supply a witness avoiding them.

\begin{proposition}\label{prop:codeg6}
Let $G$ be $7$-regular and let $uv\in E(G)$ with $\cod(uv)=6$. Then
$\{u,v\}$ is reducible.
\end{proposition}

\begin{proof}
If $G[C]\ne K_6$, apply Proposition~\ref{prop:template6} and
Lemma~\ref{lem:template} with $t=1$. If $G[C]=K_6$, apply
Lemma~\ref{lem:fano6}.
\end{proof}

Proposition~\ref{prop:codeg6} is unconditional: it uses no hypothesis on the
links, only $7$-regularity.

\section{Triangle codegree five}\label{sec:codeg5}

Here $c=5$, so $A=\{a\}$ and $B=\{b\}$ are singletons, $W$ has $9$ vertices,
and $L$ is determined by the triple
\[
  \bigl(H,\ \N(a)\cap C,\ \N(b)\cap C\bigr),\qquad H=G[C].
\]
The possible edge $ab$ is an $A$--$B$ edge and is therefore absent from $L$;
by Lemma~\ref{lem:localcomplete} this is harmless. Note that the link of $u$ is
$H_u=G[\N(u)]=G[\{v\}\cup C\cup\{a\}]$, a graph on seven vertices in which
$v$ is adjacent to all of $C$ and not to $a$, $C$ induces $H$, and $a$ is
adjacent exactly to $\N(a)\cap C$. So $H_u$ is determined by $H$ and
$\N(a)\cap C$; write $H_u=\Lambda(H,\N(a)\cap C)$. Symmetrically
$H_v=\Lambda(H,\N(b)\cap C)$.

The template now leaves a short list of exceptions, and the link structure of
a minimal counterexample removes all but one of them.

\begin{proposition}[verified exhaustively]\label{prop:template5}
Exactly $32$ of the $1{,}024$ labelled graphs $H$ on the five vertices of $C$
admit no $R\subseteq E(H)$ with $3+|R|+2q(R)\le 2p(R)$. They form five
isomorphism classes:
\[
  \overline{K_5}\ (\text{edgeless}),\quad K_{1,4},\quad K_3\sqcup K_2,\quad
  \text{the bowtie } K_1\vee 2K_2,\quad K_5,
\]
with $1+5+10+15+1=32$ labelled copies respectively.
\end{proposition}

\emph{Verification.} \texttt{tests/check\_codegree5.py} enumerates all $2^{10}$
labelled $H$ and all $R\subseteq E(H)$, computing $q$ and $p$ from their
definitions, and reports the failure set and its decomposition into
isomorphism classes under the action of $S_5$. \hfill$\square$

\begin{remark}
The bowtie $K_1\vee 2K_2$ (two triangles sharing a vertex, six edges) has $15$
labelled copies on five vertices, as does $K_1\vee C_4$ (the wheel $W_4$,
eight edges). They are distinct graphs, and the exceptional class here is the
bowtie.
\end{remark}

\begin{proposition}[verified exhaustively]\label{prop:exceptional5}
Let $H$ be one of the $31$ labelled graphs of Proposition~\ref{prop:template5}
other than $K_5$, and let $\emptyset\ne D\subseteq C$. Then
$\Lambda(H,D)\in\WKE$. If $D=\emptyset$ then $\Lambda(H,D)$ is disconnected.
\end{proposition}

\emph{Verification.} The second sentence is immediate: if $D=\emptyset$ then
the vertex $a$ is isolated in $\Lambda(H,D)$, which has seven vertices. For the
first, \texttt{tests/check\_codegree5.py} tests all $31\times 31=961$ pairs
$(H,D)$ against Definition~\ref{def:wke} directly and finds every resulting
link to be in $\WKE$. (By contrast, for $H=K_5$ all $31$ nonempty $D$ give a
connected non-$\WKE$ link, so $K_5$ must be handled separately.)
\hfill$\square$

\begin{lemma}[a Fano witness for $K_5$]\label{lem:fano5}
Let $G$ be $7$-regular and $uv\in E(G)$ with $\cod(uv)=5$ and $G[C]\cong K_5$.
Then $\{u,v\}$ is reducible.
\end{lemma}

\begin{proof}
Write $C=\{c_0,\dots,c_4\}$, $A=\{a\}$, $B=\{b\}$. Take
\[
  S=\{\,uvc_0,\ uc_1c_2,\ uc_3c_4,\ vc_1c_3,\ vc_2c_4,\ c_0c_1c_4,\
        c_0c_2c_3\,\},
  \qquad
  X=E(G[C])\cup\{uv,\ ua,\ vb\}.
\]
All members of $S$ are triangles of $L$: $u$ and $v$ are adjacent to all of
$C$ and to each other, and $G[C]=K_5$. Exactly as in the proof of
Lemma~\ref{lem:fano6}, $S$ is edge-disjoint with $|S|=7$, and
$|X|=10+3=13\le 14=2|S|$.

Condition (ii): a triangle of $L$ containing $u$ has its other two vertices in
$\N_L(u)=\{v\}\cup C\cup\{a\}$. If $v$ is one of them, it contains $uv\in X$;
if $a$ is one of them, it contains $ua\in X$; otherwise both lie in $C$ and it
contains an edge of $G[C]\subseteq X$. Note that $a$ and $v$ cannot both occur,
since $av\notin E(G)$. Symmetrically for $v$, using $vb$.

Condition (iii): every $S$-edge not incident with $\{u,v\}$ lies inside $C$,
hence in $E(G[C])\subseteq X$.
\end{proof}

\begin{proposition}\label{prop:codeg5}
Let $G$ be a $7$-regular graph satisfying the standing hypothesis of
Section~\ref{sec:minimal}, and let $uv\in E(G)$ with $\cod(uv)=5$. Then
$\{u,v\}$ is reducible.
\end{proposition}

\begin{proof}
Put $H=G[C]$. By Lemma~\ref{lem:linkstructure} the link
$H_u=\Lambda(H,\N(a)\cap C)$ is connected and not in $\WKE$. If $H$ were one
of the $31$ labelled exceptional graphs of
Proposition~\ref{prop:template5} other than $K_5$, then
Proposition~\ref{prop:exceptional5} would make $H_u$ either disconnected or a
member of $\WKE$ --- a contradiction. So either $H$ admits a template
$R$ satisfying $3+|R|+2q(R)\le 2p(R)$, and Lemma~\ref{lem:template} applies
with $t=3$; or $H\cong K_5$, and Lemma~\ref{lem:fano5} applies.
\end{proof}

Unlike Proposition~\ref{prop:codeg6}, Proposition~\ref{prop:codeg5} uses the
hypothesis that the links are connected and non-$\WKE$, and it uses it only
through Proposition~\ref{prop:exceptional5}, to exclude four core types. If
the boundary set $\N(a)\cap C$ is nonempty for one of those cores, the link is
$\WKE$ and Puleo's lemma already makes $\{u\}$ reducible; if it is empty, the
link is disconnected.

\section{Triangle codegree four}\label{sec:codeg4}

Here $c=4$, so $|A|=|B|=2$ and $W$ has $10$ vertices. Write
$C=\{c_0,c_1,c_2,c_3\}$, $A=\{a_0,a_1\}$, $B=\{b_0,b_1\}$. The local graph
$L$ is determined by the triple $(\kappa,\lambda,\rho)$ where
\begin{itemize}
  \item $\kappa$ records the $\binom{4}{2}=6$ possible edges inside $C$;
  \item $\lambda$ records the $8$ possible $A$--$C$ edges together with the
    edge $a_0a_1$, so $9$ bits; and
  \item $\rho$ records the $8$ possible $B$--$C$ edges together with $b_0b_1$,
    again $9$ bits.
\end{itemize}
The link of $u$ is $H_u=G[\{v\}\cup C\cup A]$, a graph on seven vertices
determined by $(\kappa,\lambda)$: $v$ is adjacent to all of $C$ and to no
vertex of $A$. Symmetrically $H_v$ is determined by $(\kappa,\rho)$.

For $c=4$ the template does not cover all relevant local graphs, so we
enumerate the configurations satisfying the link and degree constraints.
Call a triple $(\kappa,\lambda,\rho)$ \emph{admissible} if

\begin{enumerate}
  \item[(A1)] the links determined by $(\kappa,\lambda)$ and by
    $(\kappa,\rho)$ are both connected and not in $\WKE$; and
  \item[(A2)] the degree budget of Lemma~\ref{lem:budget} holds, i.e.\
    $\dg_\kappa(x)+\dg_\lambda(x)+\dg_\rho(x)\le 5$ for every $x\in C$, where
    $\dg_\kappa(x)$ is the number of neighbours of $x$ inside $C$ recorded by
    $\kappa$, and $\dg_\lambda(x)$, $\dg_\rho(x)$ are the numbers of its
    neighbours in $A$ and in $B$ recorded by $\lambda$ and $\rho$.
\end{enumerate}

By Lemma~\ref{lem:budget}, condition (A2) holds automatically for the triple
arising from any codegree-four edge of any $7$-regular graph; and (A1) is
exactly the hypothesis of Proposition~\ref{prop:codeg4} below, which by
Lemma~\ref{lem:linkstructure} is satisfied in a graph obeying the standing
hypothesis of Section~\ref{sec:minimal}.

Let $\Gamma=(S_4\times S_2\times S_2)\rtimes S_2$ act on triples, where $S_4$
permutes $C$ (simultaneously in $\kappa$, $\lambda$ and $\rho$), the two
copies of $S_2$ exchange $a_0\leftrightarrow a_1$ and
$b_0\leftrightarrow b_1$, and the outer $S_2$ exchanges the hubs, sending
$(\kappa,\lambda,\rho)$ to $(\kappa,\rho,\lambda)$. Every element of $\Gamma$
is just a relabelling of $C$, a relabelling within $A$ or $B$, or the exchange
of $(u,A)$ with $(v,B)$. It is therefore induced by an isomorphism of local
graphs fixing $\{u,v\}$ setwise, so admissibility and reducibility are
$\Gamma$-invariant.

\begin{proposition}[verified exhaustively; explicit certificates]
\label{prop:codeg4}
Let $G$ be a $7$-regular graph and $uv\in E(G)$ with $\cod(uv)=4$ such that
both links $G[\N(u)]$ and $G[\N(v)]$ are connected and not in $\WKE$. Then
$\{u,v\}$ is reducible.
\end{proposition}

\emph{Verification.} There are $2^{15}=32{,}768$ pairs $(\kappa,\lambda)$; of
these, exactly $4{,}667$ give a connected non-$\WKE$ link. The number of
$\Gamma$-orbits of triples satisfying (A1) is $3{,}286$; imposing (A2)
leaves exactly $1{,}144$. The per-core breakdown is:

\begin{center}
\begin{tabular}{lrrr}
\hline
core $G[C]$ & labelled sides & orbits under (A1) & orbits under (A1)+(A2)\\
\hline
$4K_1$        &   0 &   0 &   0\\
$K_2\sqcup 2K_1$ & 0 &   0 &   0\\
$P_3\sqcup K_1$ &  45 & 192 & 150\\
$K_{1,3}$     &  73 & 187 & 113\\
$K_3\sqcup K_1$ &  60 & 124 &  25\\
$2K_2$        &  21 &  17 &  17\\
$P_4$         &  76 & 430 & 273\\
paw           & 102 & 738 & 179\\
$C_4$         & 128 & 410 & 227\\
$K_4-e$       & 142 & 891 & 139\\
$K_4$         & 160 & 297 &  21\\
\hline
\textbf{total} & & \textbf{3286} & \textbf{1144}\\
\hline
\end{tabular}
\end{center}

The accompanying catalogue contains an explicit pair $(S,X)$ for each of the
$1{,}144$ canonical orbit representatives. The checker re-derives the whole
orbit structure, confirms that the certificate keys are exactly the derived
orbit set, and, for each certificate and each of the $16$ possible patterns of
$A$--$B$ edges, reconstructs the ten-vertex graph and verifies conditions
(i)--(iii) of Corollary~\ref{cor:localcriterion} by direct edge incidence.
Thus it performs $1{,}144\times16=18{,}304$ literal verifications. The
certificates were originally found by an integer-constraint search; no
property of that search is used, only its output, which is checked
independently. \hfill$\square$

Proposition~\ref{prop:codeg4} is a finite classification with a certificate
for each class, and we do not provide a human-readable proof of it. To indicate
what the certificates look like, here is one valid representative witness for the core $K_4$
with $a_0$ complete to $C$, $a_0a_1\in E$, $b_0$ complete to $C$,
$b_0b_1\in E$ and no other side edge:
\[
  S=\{\,uvc_0,\ uc_1c_3,\ uc_2a_0,\ vc_2c_3,\ vb_0b_1,\ c_0c_1c_2\,\},
\]
\[
  X=\{\,uv,\ uc_3,\ ua_0,\ vc_3,\ vb_0,\
        c_0c_1,c_0c_2,c_1c_2,c_1c_3,c_2c_3,\ c_2a_0,\ b_0b_1\,\}.
\]
Here $|X|=12=2|S|$, so there is no slack, and the packing uses the triangle
$c_0c_1c_2$, which lies wholly inside the core and meets neither hub. This is
a valid witness for the orbit with key $(63,271,271)$, but it is not the
certificate stored for that orbit; the stored record uses seven triangles.
The catalogue itself contains $217$ certificates using a triangle that avoids
both hubs, including $22$ using a triangle wholly inside $C$, so it is not
confined to the triangles available to the template of
Section~\ref{sec:template}. Every stored certificate is tight,
$|X|=2|S|$; whether witnesses with slack exist was not part of the search.

The $2{,}142$ orbits satisfying (A1) but not (A2) cannot arise from a
$7$-regular graph (Lemma~\ref{lem:budget}) and are not catalogued.

\section{Proof of Theorems~\ref{thm:main} and \ref{thm:regular}}
\label{sec:proof}

\begin{proof}[Proof of Theorem~\ref{thm:regular}]
Suppose not. By additivity of $\nu$ and $\tau$ over components, and since
every component of a $7$-regular graph is connected and $7$-regular, there is
a connected $7$-regular $G$ with $\tau(G)>2\nu(G)$; this is the standing
hypothesis of Section~\ref{sec:minimal}.

By Corollary~\ref{cor:trichotomy} there is an edge $uv\in E(G)$ with
$\cod(uv)\in\{4,5,6\}$. By Lemma~\ref{lem:linkstructure} both links
$G[\N(u)]$ and $G[\N(v)]$ are connected and not in $\WKE$. Now:
\begin{itemize}
  \item if $\cod(uv)=6$, Proposition~\ref{prop:codeg6} makes $\{u,v\}$
    reducible;
  \item if $\cod(uv)=5$, Proposition~\ref{prop:codeg5} makes $\{u,v\}$
    reducible;
  \item if $\cod(uv)=4$, Proposition~\ref{prop:codeg4} makes $\{u,v\}$
    reducible.
\end{itemize}
In every case $G$ has a reducible vertex set, contradicting
Lemma~\ref{lem:norreducible}.
\end{proof}

\begin{proof}[Proof of Theorem~\ref{thm:main}]
Proposition~\ref{prop:AtoAprime} and Theorem~\ref{thm:regular}.
\end{proof}

\section{Sharpness}\label{sec:sharp}

Tuza already observed that graphs whose blocks are copies of $K_4$, $K_5$ or
$K_2$ satisfy $\tau=2\nu$, with $\nu$ arbitrarily large
\cite[pp.~373--374]{Tuza1990}. We record the $K_4$-block calculation because
its degree bound is what is needed here.

\begin{lemma}\label{lem:k4k5}
$\nu(K_4)=1$, $\tau(K_4)=2$, $\nu(K_5)=2$ and $\tau(K_5)=4$.
\end{lemma}

\begin{proof}
$K_4$ has four triangles, and any two of them share an edge, so $\nu(K_4)=1$;
deleting one edge leaves a triangle, and deleting a perfect matching leaves
$C_4$, which is triangle-free, so $\tau(K_4)=2$. The values for $K_5$ are
also elementary. The triangles $123$ and $145$ are edge-disjoint. Three
edge-disjoint triangles would use nine edges, so their union would be
$K_5-e$; this is impossible because a union of triangles has even degree at
every vertex, whereas $K_5-e$ has two vertices of degree three. Hence
$\nu(K_5)=2$. By Mantel's theorem a triangle-free graph on five vertices has
at most six edges, so at least four edges must be deleted from $K_5$.
Conversely, deleting the four edges inside the two parts of a $2+3$ partition
leaves $K_{2,3}$, and therefore $\tau(K_5)=4$.
\end{proof}

\begin{lemma}\label{lem:k4blocks}
Let $G$ be a connected graph all of whose blocks are copies of $K_4$, and let
$k$ be the number of blocks. Then $\nu(G)=k$ and $\tau(G)=2k$.
Moreover $\Delta(G)\le 7$ if and only if no vertex of $G$ lies in more than two
blocks, in which case $\Delta(G)\le 6$.
\end{lemma}

\begin{proof}
Every triangle of $G$ is $2$-connected, hence contained in a single block, and
distinct blocks are edge-disjoint. Therefore both a maximum packing and a
minimum transversal decompose blockwise, giving
$\nu(G)=\sum\nu(K_4)=k$ and $\tau(G)=\sum\tau(K_4)=2k$ by
Lemma~\ref{lem:k4k5}. A vertex lying in exactly $r$ blocks has degree $3r$; so
$\Delta(G)\le7$ forces $r\le2$ for all vertices, and conversely $r\le 2$ gives
$\Delta(G)\le6$.
\end{proof}

\begin{proof}[Proof of Proposition~\ref{prop:sharpintro}]
$K_4$ and $K_5$ are handled by Lemma~\ref{lem:k4k5}; their maximum degrees
are $3$ and $4$. For arbitrary $k\ge1$, glue $k$ copies of $K_4$ in a
path, identifying one vertex of the $i$-th copy with one vertex of the
$(i+1)$-st, all identifications at distinct vertices. Every vertex lies in at
most two blocks, so $\Delta\le6$, and Lemma~\ref{lem:k4blocks} gives
$\nu=k$, $\tau=2k$; call this graph $G_6$.

Pendant edges lie in no triangle and therefore change neither $\nu$ nor
$\tau$. If $k\ge2$, attach one pendant edge at a cut vertex of $G_6$, which
has degree six; the resulting connected graph $G_7$ has maximum degree seven.
If $k=1$, attach four pendant edges to one vertex of $K_4$. This again gives a
connected graph $G_7$ with maximum degree seven and $(\nu,\tau)=(1,2)$.
Thus the two claimed families exist. Since $\tau=2\nu$ on these graphs, no bound
$\tau\le c\nu$ with $c<2$ and no strict bound $\tau<2\nu$ can hold on the
class $\Delta\le 7$.
\end{proof}

Theorem~\ref{thm:main} is therefore best possible as an inequality. The
connected examples attain equality for every positive value of $\nu$.
Consequently an inequality $\tau\le2\nu-f(\nu)$ cannot hold on this class if
$f(k)>0$ for even one positive integer $k$.
For planar graphs, Cui, Haxell and Ma characterized all graphs attaining
$\tau=2\nu$ \cite{CuiHaxellMa2009}; the examples above are planar instances
of their extremal class.

\section{Concluding remarks}\label{sec:open}

The proof uses non-$\WKE$ links twice. First, the seven-vertex census forces
an incident edge of codegree at least four. Second, the local-transfer lemma
allows reducibility of that edge's endpoints to be checked in a graph of at
most ten vertices. The first step suggests looking for edge-pair reductions
beyond degree seven, and the second remains valid without change. The finite
classifications do not.

\begin{question}\label{q:mad8}
Does $\mad(G)<8$ imply $\tau(G)\le2\nu(G)$?
\end{question}

Theorem~\ref{thm:main} settles the maximum-degree-seven portion of the boundary
$\mad(G)\le7$, but not graphs of maximum average degree seven that have
vertices of larger degree. The same first step suggests asking whether the
stronger sparsity statement above is tractable. The finite classifications of
Sections~\ref{sec:codeg6}--\ref{sec:codeg4} depend throughout on
$7$-regularity: through the degree budget of Lemma~\ref{lem:budget}, the
identity $|A|=|B|=6-c$, and the seven-vertex link census of
Lemma~\ref{lem:B1}. Every one of them would have to be redone, and we make no
prediction about whether the method survives.

\begin{question}\label{q:human}
Is there a human-readable proof of Proposition~\ref{prop:codeg4}?
\end{question}

Sections~\ref{sec:codeg6} and \ref{sec:codeg5} are structured proofs whose
computational content is the verification of one template inequality over a
small explicit family, together with a short exceptional case settled by a
Fano configuration. Section~\ref{sec:codeg4} has no comparable structural
argument: it is a catalogue of $1{,}144$ certificates. A conceptual proof of
Proposition~\ref{prop:codeg4} would be the natural starting point for an
attack on Question~\ref{q:mad8}.

\begin{question}\label{q:equality}
Which graphs of maximum degree at most seven satisfy $\tau=2\nu$?
\end{question}

The planar case has the exact characterization of Cui, Haxell and Ma
\cite{CuiHaxellMa2009}. It would be interesting to know what replaces their
$K_4$-based extremal class without planarity, and in particular how much the
$K_5$ blocks in Tuza's general equality family contribute. In particular,
does any $7$-regular graph satisfy $\tau=2\nu$? The finite checks in this paper
do not address that question.

\begin{question}\label{q:k5free}
Does every $K_5$-free graph satisfy $\tau\le2\nu$?
\end{question}

If $G$ is $K_5$-free and an edge has at least three common neighbours, then
its core $G[C]$ is triangle-free: a triangle in $C$, together with the two
hubs, would form a $K_5$. In particular the two complete cores --- the only
configurations at codegrees five and six that defeat the hub-only template ---
cannot occur. At codegree four it also removes all four core types containing
a triangle, accounting for $364$ of the $1{,}144$ catalogued orbits. Whether
that observation can replace the degree hypothesis is unclear. The stronger
restriction to $K_4$-free graphs is itself unresolved in general; work of
Munaro treats several subclasses and shows that the factor two would remain
essentially tight even there \cite{Munaro2018}.

Tuza's conjecture remains open, and Theorem~\ref{thm:main} does not improve
Haxell's universal bound $\tau\le(66/23)\,\nu$ \cite{Haxell1999}. Its narrower
point is that the equality boundary of Puleo's sparsity theorem can be crossed,
which motivates Question~\ref{q:mad8}.

\section*{Data and code availability}
The certificate catalogue and verification programs are included with the
arXiv submission as ancillary files and are maintained in the companion
repository at
\url{https://github.com/agupta/tuza-maximum-degree-seven}.

\section*{Acknowledgment of generative-AI assistance}
Anthropic Claude Code (Claude 5 family) and OpenAI Codex (GPT-5.6 family) were
used extensively for proof exploration, software development, exact
computational checks, literature discovery, and drafting and editing the
manuscript. The author selected the arguments and methods, checked the cited
sources and reported computations, and takes full responsibility for the
content. These systems are not authors or independent guarantors of
correctness.

\appendix

\section{The verification programs}\label{app:programs}

The four checkers in \texttt{tests/} use Python~3.10 or later and only its
standard library. They use exact integer and Boolean arithmetic throughout;
there is no floating point, no numerical tolerance and no external solver in
the verification path. Each prints a line beginning
\texttt{PASS} for every assertion it discharges and exits with status $0$ if
and only if all of them hold. They take no arguments, read no input other
than the certificate file, and use no network. The optional finder in
\texttt{src/} uses OR-Tools CP-SAT; it is not part of the verification path.

\medskip
\noindent\textbf{\texttt{tests/check\_codegree4.py}}
verifies Lemma~\ref{lem:B1} by enumerating all $2^{21}$ labelled
seven-vertex graphs; derives the $4{,}667$ admissible
labelled core/side pairs, the $3{,}286$ orbits under (A1) and the $1{,}144$
under (A1)+(A2); checks that the keys of
\path{data/B_codegree4_certificates.json} are exactly the derived orbit
set; and verifies all $1{,}144$ certificates under all $16$ ambient
$A$--$B$ patterns, i.e.\ $18{,}304$ literal checks of conditions (i)--(iii)
of Corollary~\ref{cor:localcriterion}. It also asserts the reported certificate
statistics and checks the displayed witness for key $(63,271,271)$ under all
$16$ patterns, separately from the stored record. The ``labelled sides'' column in the
table for Proposition~\ref{prop:codeg4} counts, for one fixed labelled core
representative, the side masks giving a connected non-$\WKE$ link; summing
over all $64$ labelled cores with multiplicity gives $4{,}667$.

The certificate file is self-describing. Its vertex labels are
$0=u$, $1=v$, $2,3,4,5=C$, $6,7=A$ and $8,9=B$. Each record gives the three
bit masks $(\kappa,\lambda,\rho)$ together with a list \texttt{packing} of
triangles and a list \texttt{cover} of edges. The checker reconstructs the
local graph from those masks, verifies that the packing is edge-disjoint and
that the cover satisfies the budget, and then checks the two incidence
conditions of Corollary~\ref{cor:localcriterion}. The JSON header records the
bit order used for the core and side masks.

\medskip
\noindent\textbf{\texttt{tests/check\_codegree5.py}}
verifies Propositions~\ref{prop:template5} and
\ref{prop:exceptional5}. It also checks every ordered boundary pair for each of
the $548$ labelled cores admitting a connected non-$\WKE$ link, with the
ambient edge $ab$ both absent and present: $548\times32\times32\times2=
1{,}122{,}304$ local states.

\medskip
\noindent\textbf{\texttt{tests/check\_codegree6.py}}
verifies Proposition~\ref{prop:template6} over all $2^{15}$ labelled cores
and all $R\subseteq E(H)$ --- including the failure at $K_6$, the intrinsic
$155$-core tightness count, and the Fano witness of Lemma~\ref{lem:fano6}.

\medskip
\noindent\textbf{\texttt{tests/check\_sharpness.py}}.
Recomputes $\nu$ and $\tau$ by complete search for the graphs named in
Section~\ref{sec:sharp}, together with the additional dense anchors $K_7$ and
$K_8$. This script is a reference check; no step of the proof depends on it.

\medskip
\noindent\textbf{\texttt{src/B\_codegree4\_search.py}}.
The program that originally \emph{found} the codegree-four certificates. It is
supplied as historical provenance, not as a supported reproducible
environment: no OR-Tools version is pinned. Nothing in the proof depends on
it. Only the emitted certificates matter, and those are re-verified by
\texttt{tests/check\_codegree4.py} without reference to how they were
produced.

\medskip
The checkers assert that their live derivations reproduce hard-coded expected
data: three tables in \path{tests/check_codegree4.py}, two expected
collections in \path{tests/check_codegree5.py}, and one best-margin
distribution in \path{tests/check_codegree6.py}. These literals guard
against drift; they are not independent re-derivations. In the codegree-four
checker the total orbit count and exact certificate coverage are derived
without hard-coded per-core input. The checkers verify the finite local statements
(Lemma~\ref{lem:B1}, Propositions~\ref{prop:template6},
\ref{prop:template5}, \ref{prop:exceptional5} and \ref{prop:codeg4}); they do
not verify Theorem~\ref{thm:regular} or Theorem~\ref{thm:main}, whose
assembly in Section~\ref{sec:proof} is a human argument.

\raggedright


\begin{thebibliography}{99}

\bibitem{AparnaBujtasTuza2012}
S.~Aparna Lakshmanan, Cs.~Bujt\'as, Zs.~Tuza,
\emph{Small edge sets meeting all triangles of a graph},
Graphs Combin.\ \textbf{28} (3) (2012) 381--392;
doi:10.1007/s00373-011-1048-8.

\bibitem{BotlerFG2021}
F.~Botler, C.~G.~Fernandes, J.~Guti\'errez,
\emph{On Tuza's conjecture for triangulations and graphs with small
treewidth}, Discrete Math.\ \textbf{344} (4) (2021), 112281;
doi:10.1016/j.disc.2020.112281; arXiv:2002.07925.

\bibitem{CuiHaxellMa2009}
Q.~Cui, P.~E.~Haxell, W.~Ma,
\emph{Packing and covering triangles in planar graphs},
Graphs Combin.\ \textbf{25} (6) (2009) 817--824;
doi:10.1007/s00373-010-0881-5.

\bibitem{Haxell1999}
P.~E.~Haxell,
\emph{Packing and covering triangles in graphs},
Discrete Math.\ \textbf{195} (1--3) (1999) 251--254;
doi:10.1016/S0012-365X(98)00183-6.

\bibitem{KahnPark2022}
J.~Kahn, J.~Park,
\emph{Tuza's conjecture for random graphs},
Random Structures Algorithms \textbf{61} (2) (2022) 235--249;
doi:10.1002/rsa.21057; arXiv:2007.04351.

\bibitem{Krivelevich1995}
M.~Krivelevich,
\emph{On a conjecture of Tuza about packing and covering of triangles},
Discrete Math.\ \textbf{142} (1--3) (1995) 281--286;
doi:10.1016/0012-365X(93)00228-W.

\bibitem{Munaro2018}
A.~Munaro,
\emph{Triangle packings and transversals of some $K_4$-free graphs},
Graphs Combin.\ \textbf{34} (4) (2018) 647--668;
doi:10.1007/s00373-018-1903-y.

\bibitem{Puleo2015}
G.~J.~Puleo,
\emph{Tuza's conjecture for graphs with maximum average degree less than 7},
European J.\ Combin.\ \textbf{49} (2015) 134--152;
doi:10.1016/j.ejc.2015.03.006; arXiv:1308.2211v3.

\bibitem{Tuza1981}
Zs.~Tuza,
\emph{Conjecture}, in: A.~Hajnal, L.~Lov\'asz, V.~T.~S\'os (Eds.),
Finite and Infinite Sets (Proc.\ Sixth Hungarian Combinatorial Colloquium,
Eger, 1981), Colloq.\ Math.\ Soc.\ J\'anos Bolyai \textbf{37},
North-Holland, Amsterdam, 1984, p.~888.

\bibitem{Tuza1990}
Zs.~Tuza,
\emph{A conjecture on triangles of graphs},
Graphs Combin.\ \textbf{6} (4) (1990) 373--380;
doi:10.1007/BF01787705.

\end{thebibliography}
\end{document}